\documentclass[11pt]{amsart}

\usepackage{amsmath, amsthm, amssymb}
\usepackage{enumitem}
\usepackage{parskip}%No Indent
  
\usepackage{lmodern}
\usepackage{palatino}                       % text font
\usepackage[bigdelims,vvarbb]{newpxmath}    % math font
\usepackage[scaled=0.95]{inconsolata}       % Sans font 
\usepackage[T1]{fontenc}                    % font encoding 

\usepackage{comment}

\usepackage[abs]{overpic}		  % overpic automatically loads graphicx
\usepackage[all,cmtip]{xy}

\usepackage{xcolor}
\definecolor{indigo}{rgb}{0.29, 0.0, 0.51}  % custom colors
\usepackage[colorlinks, urlcolor=indigo, linkcolor=indigo, citecolor=indigo]{hyperref}

\usepackage[hcentering, vcentering, total={5.9in, 8.2in}]{geometry}  % 1.4 vertical margin; 1.3 horizontal margin

\usepackage{tikz}
\usetikzlibrary{trees}

\usepackage{mathtools}

\theoremstyle{plain}
\newtheorem{theorem}{Theorem}
\newtheorem{corollary}[theorem]{Corollary}

\newtheorem{lemma}[theorem]{Lemma}

\theoremstyle{definition}

\theoremstyle{remark}
\newtheorem{remark}[theorem]{Remark}

\numberwithin{theorem}{section}

\newcommand{\R}{\mathbb{R}}           % the real numbers
\newcommand{\NS}{{\mathbb{S}}}

\newcommand{\D}{{\mathbb{D}}}

\newcommand{\Op}{{\mathcal{O}p}}

\makeatletter
\newcommand*\bigcdot{\mathpalette\bigcdot@{0.6}}
\newcommand*\bigcdot@[2]{\mathbin{\vcenter{\hbox{\scalebox{#2}{$\m@th#1\bullet$}}}}}
\makeatother

\DeclareFontFamily{U} {cmr}{}
\DeclareFontShape{U}{cmr}{m}{n}{
  <-6> cmr5
  <6-7> cmr6
  <7-8> cmr7
  <8-9> cmr8
  <9-10> cmr9
  <10-12> cmr8
  <12-> cmr9}{}
\DeclareSymbolFont{Xcmr} {U} {cmr}{m}{n}
\DeclareMathSymbol{\Phis}{\mathord}{Xcmr}{8}

\newcommand{\rel}{\operatorname{rel}}

\newcommand{\Id}{{\operatorname{Id}}}

\newcommand{\std}{\operatorname{std}}

\newcommand{\Leg}{\operatorname{Leg}}
\newcommand{\FLeg}{\operatorname{FLeg}}
\newcommand{\DB}{\operatorname{DB}}
\newcommand{\Cont}{\operatorname{Cont}}

\newcommand{\CFr}{\operatorname{CFr}}
\newcommand{\ev}{\operatorname{ev}}

\newcommand{\image}{\operatorname{Image}}

\newcommand{\C}{\operatorname{C}}

\newcommand{\OT}{\operatorname{OT}}

\newcommand{\Emb}{{\operatorname{Emb}}}

\newcommand{\Bypass}{\operatorname{\mathcal{B}}}

\begin{document} 

% title
\title{A Bypass in the Middle}

\author{Dahyana Farias}
\address{Departamento de Matem\'aticas \\ Universidad Aut\'onoma de Madrid \\ Madrid \\ Espa\~na}
\email{dahyana.farias@estudiante.uam.es \\ dahyana.farias@icmat.es}

\author{Eduardo Fern\'{a}ndez}
\address{Department of Mathematics\\ University of Georgia\\ Athens\\ GA, USA}
\email{eduardofernandez@uga.edu}

%\subjclass[]{}

\maketitle

\begin{abstract}
We provide a topological characterization for a family of bypasses with a fixed attaching arc to be contractible. This characterization is formulated in terms of the existence of a bypass that is disjoint from the given family away from the attaching region. As an application, we provide new proofs of several $h$-principles in overtwisted contact $3$-manifolds.
\end{abstract}

%\setcounter{tocdepth}{1} 
%\tableofcontents

%%%%%%%%%%%%%%%%%%%%%%%%%%%%%%%%%%%%%%%
\section{Introduction}\label{sec:intro}
%%%%%%%%%%%%%%%%%%%%%%%%%%%%%%%%%%%%%%%

\subsection{Context} Since their inception by Ko Honda \cite{Honda}, bypass disks have been a central object in the classification problem for both contact structures in a given smooth 3-manifold and Legendrians in a given contact 3-manifold. Determining the existence or non-existence of certain bypasses is a crucial aspect of the general strategies employed in these classifications (e.g., \cite{EtnyreHonda,Honda,HondaGluing}).

One of the recent questions that has garnered attention is the classification of families of contact structures and Legendrian embeddings. This interest arises, among other things, from the fact that classifying 1-parameter families of contact structures is closely related to classifying contactomorphisms up to contact isotopy, which is key to understanding the Legendrian isotopy problem outside the standard contact 3-sphere, and that Legendrian loops can be used to construct interesting Lagrangian fillings or higher-dimensional Legendrians (see, for example, \cite{CasalsGao,DingGeiges,EkholmKalman,EliashbergMishachev,FMP22,FMP24,FernandezMin,FernandezMuñoz,GirouxSurfaces,GirouxMassot,Kalman,Min,VogelOvertwisted}). As a consequence of the Fuchs-Tabachnikov theorem \cite{CasalsDelPino,FuchsTabachnikov,Murphy}, this classification can be reduced to studying the homotopy groups of certain bypass embedding spaces, i.e., (families of) bypass embeddings up to homotopy among (families of) bypass embeddings. For Legendrians, this reduction follows from the fact that bypasses can be used to destabilize Legendrians. For contact structures, this follows by attempting to mimic Giroux's construction of adapted open books \cite{GirouxICM} in a parametric fashion combined with \cite{FMP22}.

Motivated by these, we study bypass embedding spaces. Our main result shows that the only potentially interesting behavior in the contact isotopy problem for bypass embeddings arises when intersections between bypasses are present. For experts, this parallels the contact isotopy problem for overtwisted disks: it is topological in nature when there are no intersections \cite{Dymara}, but can become highly non-trivial when intersections are involved \cite{VogelOvertwisted}.

\subsection{Main result}

Before stating our main result, we need to introduce some notation. Let $(M, \xi)$ be a contact 3-manifold and $\Sigma \subseteq (M, \xi)$ a convex surface. A \em bypass embedding \em is an embedding $b: \D^2_{+} \rightarrow (M, \xi)$ of the half-disk $\D^2_+ = \D^2 \cap {y \geq 0} \subseteq \R^2$ with Legendrian boundary whose characteristic foliation matches the bypass characteristic foliation $\mathcal{F}_{B}$ depicted in Figure \ref{fig:bypass1}. In other words, the foliation on $\D^2_+$ determined by $db^{-1}(db(T\D^2_+) \cap \xi)$ is exactly $\mathcal{F}_{B}$. The attaching arc of the bypass is the Legendrian arc $\gamma =\{b(x, 0) : -1\leq x \leq 1\}$.

The bypass embedding $b$ is said to be a bypass embedding for $\Sigma$ if $b(\D^2_+) \cap \Sigma = \gamma$. Given another bypass embedding $\hat{b}$ with the same attaching arc $\gamma$, we say that $\hat{b}$ coincides with $b$ near the attaching arc if there exists an open neighborhood $U \subseteq \D^2_+$ of $\{(x, 0) : -1 \leq x \leq 1\}$ such that $b_{|\overline{U}} = \hat{b}_{|\overline{U}}$. If this neighborhood can be chosen such that $b(\D^2{+} \backslash U) \cap \hat{b}(\D^2_+ \backslash U) = \emptyset$, then we say that $\hat{b}$ is disjoint from $b$ away from the attaching arc. Note that this involves a slight abuse of terminology: $b$ is considered disjoint from itself away from the attaching arc. A schematic of two bypasses disjoint away from the attaching arc is shown in Figure \ref{fig:middle1}.

We denote by $\mathcal{B}_\gamma(\Sigma, (M, \xi))$ the space of bypass embeddings for $\Sigma$ with a fixed attaching arc $\gamma$ that coincide with $b$ near the attaching arc, and by $\mathcal{B}_\gamma^b(\Sigma, (M, \xi)) \subseteq \mathcal{B}_\gamma(\Sigma, (M, \xi))$ the subspace of bypass embeddings that are disjoint from $b$ away from the attaching arc.

Our main result provides a necessary and sufficient condition for a family of bypass embeddings with fixed attaching arc to be contractible:

\begin{theorem}[Bypass in the Middle]\label{teo:bypassMiddle}
     The inclusion $\mathcal{B}_\gamma^b(\Sigma,(M,\xi))\hookrightarrow \mathcal{B}_\gamma(\Sigma,(M,\xi))$ is null-homotopic.
\end{theorem}

\begin{remark}
    Let $b_0,b_1\in\mathcal{B}_\gamma (\Sigma,(M,\xi))$ be two bypass embeddings for $\Sigma$ that coincide near the attaching arc $\gamma$. The result above implies that if there exists a third bypass embedding $b\in\mathcal{B}_\gamma (\Sigma,(M,\xi))$ coinciding with both $b_0$ and $b_1$ near the attaching arc, but becoming disjoint from both $b_0$ and $b_1$ away from the attaching arc, from now on referred as \em bypass in the middle, \em then there exists a homotopy of bypass embeddings connecting $b_0$ and $b_1$. In particular, if $b_0$ and $b_1$ are disjoint away from $\gamma$ then they are isotopic as bypass embeddings. Hence, potentially interesting pairs of bypass embeddings \em must \em intersect each other. This situation is completely analogous to the one for overtwisted disk embeddings: two overtwisted disk embeddings disjoint from a third overtwisted disk are contact isotopic; and two overtwisted disks that intersect may not be contact isotopic \cite{Dymara,VogelOvertwisted}. 
\end{remark}

\begin{remark}
    The existence of the bypass in the middle $b$ in Theorem \ref{teo:bypassMiddle} is fundamental: the space $\mathcal{B}_\gamma(\Sigma,(M,\xi))$ is not even connected in general. For instance, the standard neighborhood of the stabilization of one of the Chekanov-Eliashberg knots \cite{Chekanov} admits two non contact isotopic bypasses with the same attaching arc. We depicted the boundary of both bypasses in Figure
    \ref{fig:chekanov3}. These bypasses are distinguished by means of Legendrian contact homology \cite{Chekanov}. 
    \begin{figure}
        \centering
        \includegraphics[scale=0.4]{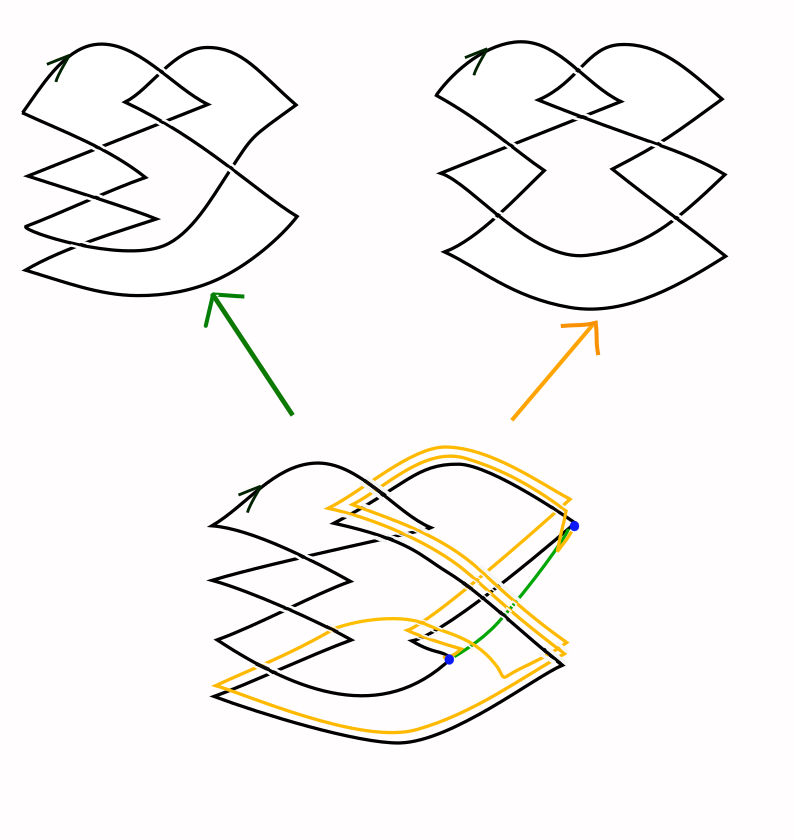}
        \caption{In the second row it is depicted a Legendrian together with the boundary of two bypasses (in green and orange) attached to its standard neighborhood. These bypasses are different because they destabilize the Legendrian into the two Chekanov-Eliashberg $m(5_2)$ knots depicted in the first row which are not Legendrian isotopic.}
        \label{fig:chekanov3}
    \end{figure}
    
\end{remark}

Our proof is elementary in the sense that it does not make use of any parametric statement in $3$-dimensional contact topology, e.g. \cite{EliashbergMishachev,FMP22}. It crucially uses the fact that a trivial bypass is indeed trivial \cite{HondaGluing,HondaHuangBypass}, which can be proved without making use of Eliashberg's uniqueness of tight contact structures in the $3$-sphere \cite{EliashbergTwenty}, and a parametric version of this: the space of trivial bypasses is contractible (Theorem \ref{thm:TrivialBypasses}). We decided to avoid using \cite{EliashbergMishachev,FMP22} in the proof so it could be potentially upgraded to the higher-dimensional case.

\subsection{Applications to overtwisted contact $3$-manifolds}\label{subsec:Applications}

To show the scope of the applicability of the Bypass in the Middle Theorem \ref{teo:bypassMiddle} we will use it to (re)prove many folk $h$-principle type results for overtwisted contact $3$-manifolds. The novelty is that our proofs just depend on Theorem \ref{teo:bypassMiddle} and do not make use of the general overtwisted $h$-princple of Eliashberg \cite{EliashbergOT}. 

Our first result concerns the contractibility of the space of bypasses in the complement of an overtwisted disk:

\begin{corollary}\label{coro:middle-OT}
     Let $(M,\xi)$ be an overtwisted contact $3$-manifold, $\Delta_{\OT}\subseteq (M,\xi)$ an overtwisted disk and $\Sigma\subseteq (M\backslash \Delta_{\OT},\xi)$ a convex surface with Legendrian boundary. Then, for any given attaching arc $\gamma\subseteq \Sigma$ the space of bypass embeddings $\mathcal{B}_\gamma(\Sigma,(M\setminus\Delta_{\OT},\xi))$ is contractible.
\end{corollary}

An \em overtwisted disk embedding \em is a smooth embedding $e:\D^2\rightarrow (M,\xi)$ of an overtwisted disk; i.e., its induced characteristic foliation coincides with that of the overtwisted disk depicted in Figure \ref{fig:disk-OT}. The space of embeddings of overtwisted disk embeddings, denoted by $\Emb^{\OT}(\D^2,(M,\xi))$, is specially relevant for the study of the homotopy type of contactomorphism group of $(M,\xi)$, see \cite{Dymara, F-SOT,VogelOvertwisted}. Denote by $\CFr(M,\xi)$ the total space of the bundle of contact frames of a contact $3$-manifold $(M,\xi)$. Since every overtwisted disk can be understood as a pair of disjoint bypasses smoothly glued along their attaching arc, we will deduce the following parametric $h$-principle for overtwisted disks in the complement of a fixed overtwisted disk from the Bypass in the Middle Theorem:

\begin{corollary}\label{cor:OvertwistedDisks}
   Let $(M,\xi)$ be an overtwisted contact $3$-manifold and $\Delta_{\OT}\subseteq (M,\xi)$ an overtwisted disk. Then, the $1$-jet evaluation map at the origin 
   $$ \ev_0:\Emb^{\OT}(\D^2,(M\backslash\Delta_{\OT},\xi))\rightarrow \CFr(M\backslash \Delta_{\OT},\xi) $$
   is a weak homotopy equivalence.
\end{corollary}

In particular, the previous Corollary implies that two overtwisted disks are contact isotopic if there is a third overtwisted disk disjoint from both. This was previously shown by Dymara \cite{Dymara} in any overtwisted $3$-sphere and for a general overtwisted contact $3$-manifold follow easily from the ideas in \cite{F-SOT}. Analogously to the situation for Theorem \ref{teo:bypassMiddle} the condition of avoiding a fixed overtwisted disk cannot be removed in general: the space of overtwisted disk embeddings could be disconnected as shown by Chekanov and Vogel independently, see \cite{VogelOvertwisted}. On the other hand, a parametric $h$-principle for the whole space of overtwisted disks holds for the subclass of strongly overtwisted contact $3$-manifolds and also for the subspace of overtwisted disks with fixed boundary \cite{F-SOT}.

\begin{figure}[h]
    \centering
    \includegraphics[scale=0.3]{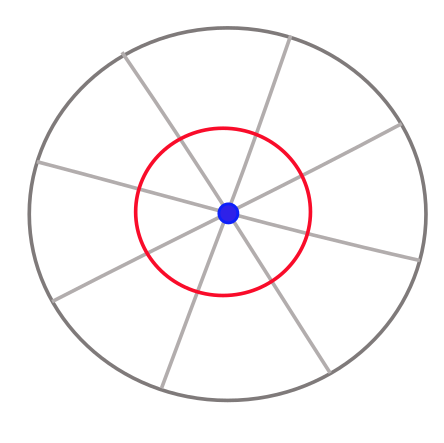}
    \caption{Overtwisted disk. The dividing set is depicted in red and the characteristic foliation in gray.}
    \label{fig:disk-OT}
\end{figure}

Finally, we will use the Bypass in the Middle Theorem \ref{teo:bypassMiddle} to study families of loose Legendrians in overtwisted contact $3$-manifolds, that is Legendrians with overtwisted complement. The space of Legendrian embeddings into a contact $3$-manifold $(M,\xi)$ is denoted by $\Leg(M,\xi)$.

\begin{theorem}\label{thm:StabilizationMapsLooseLegendrians}
Let $(M,\xi)$ be an overtwisted contact $3$-manifold and $\Delta_{\OT}\subseteq (M,\xi)$ an overtwisted disk. Then, for every $t\in\NS^1$ the stabilization at time $t$ operators 
$$ S^t_+:\Leg(M\backslash \Delta_{\OT},\xi)\rightarrow \Leg(M\backslash \Delta_{\OT},\xi) $$
 and 
 $$ S^t_-:\Leg(M\backslash \Delta_{\OT},\xi)\rightarrow \Leg(M\backslash \Delta_{\OT},\xi) $$
are weak homotopy equivalences.
\end{theorem}

Here, the stabilization operators are defined as Legendrian satellite maps following \cite{FMP24}. These operators, introduced in Section \ref{subsec:stabilizations}, act on each Legendrian by adding a positive (or negative) stabilization at time $t$, up to Legendrian isotopy.

Recall that a \em formal Legendrian embedding \em into a contact $3$-manifold $(M,\xi)$ is a pair $(\beta, F_s)$ where 
\begin{itemize}
    \item $\beta:\NS^1\hookrightarrow M$ is a smooth embedding,
    \item $F_s:T\NS^1\rightarrow \beta^*TM$, $s\in[0,1]$, is a homotopy of bundle monomorphisms such that $F_0=d\beta$ and $\image(F_1)\subseteq \beta^*\xi$.
\end{itemize}

The space of formal Legendrian embeddings into $(M,\xi)$ is denoted by $\FLeg(M,\xi)$. We will use Theorem \ref{thm:StabilizationMapsLooseLegendrians} together with the parametric Fuchs-Tabachnikov Theorem \cite{CasalsDelPino,FuchsTabachnikov,Murphy} to show that

\begin{corollary}\label{cor:LooseLegendrians}
    Let $(M,\xi)$ be an overtwisted contact $3$-manifold and $\Delta_{\OT}\subseteq (M,\xi)$ an overtwisted disk. Then, the natural inclusion 
    $$ \Leg(M\backslash\Delta_{\OT},\xi)\hookrightarrow \FLeg(M\backslash \Delta_{\OT},\xi) $$
    is a weak homotopy equivalence.
\end{corollary}

This result can be proved (in all dimensions) by using the $h$-principle for overtwisted contact structures \cite{BEM,EliashbergOT}, e.g. \cite{CardonaPresas,EliashbergCieliebak}. Notice that the result cannot be improved in general: there exist loose Legendrians which are formally isotopic but not Legendrian isotopic \cite{VogelOvertwisted}. On the other hand, if the underlying $3$-manifold is strongly overtwisted two formally isotopic loose Legendrians are Legendrian isotopic and there is a complete $h$-principle for the subclass of strongly loose Legendrians \cite{F-SOT}.

\subsection{Outline} The article is organized as follows. In Section \ref{sec:Prelim}, we review the basic concepts of contact topology that will be used throughout this manuscript. In Section \ref{sec:BypassInTheMiddle}, we prove that the space of trivial bypasses is contractible (Theorem \ref{thm:TrivialBypasses}) and use this result to establish the main theorem, Theorem \ref{teo:bypassMiddle}. Finally, in Section \ref{sec:Applications}, we prove all the applications of Theorem \ref{teo:bypassMiddle} for overtwisted contact 3-manifolds mentioned in the Introduction.

\textbf{Acknowledgments:} We are grateful to \'Alvaro del Pino and Fran Presas for insightful discussions over the years, as well as to John Etnyre, Fabio Gironella, and Hyunki Min for their interest, questions, and crucial comments on this work. The second-named author would like to thank Emmy Murphy for a helpful conversation regarding this work, and Gordana Mati\'c for her valuable comments and constant encouragement. During part of this project, EF was partially supported by an AMS-Simons Travel Grant.

%%%%%%%%%%%%%%%%%%%%%%%%%%%%%%%%%%%%%%%
\section{Preliminaries}\label{sec:Prelim}
%%%%%%%%%%%%%%%%%%%%%%%%%%%%%%%%%%%%%%%

In this Section, we review the relevant concepts from contact topology that will be used throughout the article. In \ref{subsec:Giroux}, we briefly review convex surface theory, and in \ref{subsec:bypasses}, we discuss bypasses in contact $3$-manifolds. In \ref{subsec:stabilizations}, we recall the relationship between bypasses and stabilizations, as well as the formal definition of the stabilization operator and the Fuchs-Tabachnikov result. Finally, in \ref{subsec:contactHamiltonian}, we review contact Hamiltonians. For a more detailed account of these topics, the reader is referred to \cite{Geiges:Book,Honda,Massot:Notes}.

%%%%%%%%%%%%%%%%%%%%%%%%
\subsection{Convex surface theory}\label{subsec:Giroux}
%%%%%%%%%%%%%%%%%%%%%%%

Recall that a \em contact vector field \em in a contact $3$-manifold $(M,\xi)$ is a vector field whose flow preserves the contact structure; i.e. is given by contactomorphisms. An embedded surface $\Sigma \subseteq (M,\xi)$ is said to be \em convex \em if there exists a contact vector field $X$ transverse to it. Similarly, an embedding $e:\Sigma\rightarrow (M,\xi)$ is said to be \em convex \em if $e(\Sigma)$ is convex. By \em Giroux Genericity Theorem \em every surface (with possibly non-empty Legendrian boundary) can be $C^\infty$-approximated ($C^0$-approximated near the boundary and relative to the boundary) by a convex surface. 

The \em characteristic foliation \em of $\Sigma\subseteq (M,\xi)$ is the oriented singular $1$-dimensional foliation $\mathcal{F}$ determined by $T\Sigma\cap \xi$. For embeddings we will also denote by $\mathcal{F}$ to the pull-back of the foliation determined by $Te(\Sigma)\cap \xi$. The \em dividing set \em of a convex surface $\Sigma\subseteq (M,\xi)$ associated to the contact vector field $X$ is given by a family of embedded transverse curves $\Gamma_X:=\{p\in\Sigma: X(p)\in\xi_p\}$. Since the space of contact vector fields transverse to $\Sigma$ is contractible the isotopy class of $\Gamma_X$ is independent of $X$ and we will simply write $\Gamma$ instead of $\Gamma_X$. The convexity of $\Sigma\subseteq (M,\xi)$ is equivalent to the existence of an \em $I$-invariant neighborhood \em of $\Sigma$; i.e. contactomorphic to $(\Sigma\times I,\xi_{inv}=\ker(fdt+\beta))$, where $t\in I:=[0,1]$, $f\in C^\infty(\Sigma)$ and $\beta\in \Omega^1(\Sigma)$. In this coordinates, the dividing set of $\Sigma$ is simply $\Gamma=f^{-1}(0)$.

The contact condition implies that the characteristic foliation of a surface completely determined the contact germ near the surface: 

\begin{theorem}[Giroux]\label{teo:giroux-entorno}
    Let $\Sigma_i\subseteq (M_i,\xi_i)$ be an embedded compact surface in contact 3–manifold with characteristic foliation $\mathcal{F}_i$, $i\in \{0,1\}$.
    Assume that there is a diffeomorphism $\phi: \Sigma_0 \to \Sigma_1$ such that $\phi_* \mathcal{F}_0=\mathcal{F}_1$.
    Then, there is a contactomorphism $\Phi: (\Op(\Sigma_0),\xi_0) \to (\Op(\Sigma_1),\xi_1)$  with $\Phi(\Sigma_0) = \Sigma_1$ and such that 
    $\Phi_{|\Sigma_0}$ is isotopic to $\phi$ via an isotopy preserving the characteristic foliation.   
\end{theorem}

Moreover, when the surface is also convex the contact germ is determined, after a small isotopy, by the dividing set, this is known as \em Giroux Realization Theorem. \em To state this result let us first introduce some notation. We write $\Emb(\Sigma, (M,\xi), \mathcal{F})$ to denote the space of embeddings $e : \Sigma\to M$ with characteristic foliation $\mathcal{F}$. Similarly, we will denote by $\Emb(\Sigma, (M, \xi), \Gamma)$ is the space of convex embeddings $e : \Sigma\to M$ with dividing set $\Gamma$. 

\begin{theorem}[Giroux Realization Theorem \cite{GirouxSurfaces}]\label{teo:realizacion-Giroux}
    Let $e:\Sigma\to (M,\xi)$ be a convex embedding with characteristic foliation $\mathcal{F}$ and dividing set $\Gamma$. Then, the natural inclusion
    $i:\Emb(\Sigma, (M,\xi), \mathcal{F})\to \Emb(\Sigma, (M, \xi), \Gamma)$
    is a weak homotopy equivalence.
\end{theorem}

A particular instance of realization that we will use is the following result due to Fraser.  
\begin{lemma}[Pivot Lemma]\label{lem:pivot}
Let $e: \D^2 \rightarrow (M, \xi)$ be an embedding with a characteristic foliation that consists of a single positive elliptic singularity $p$ and unstable orbits from $p$ that exit transversely through $e(\partial \D^2)$. Let $\delta_1$ and $\delta_2$ be two unstable orbits meeting at $p$, and $\delta_i \cap e(\partial \D^2) = p_i$. Then, after a $C^{\infty}$–small perturbation of $e$ that fixes $\partial \D^2$, there exists an embedding $e'$ whose characteristic foliation has exactly one positive elliptic singularity $p'$ and unstable orbits from $p'$ that exit transversely through $\partial \D^2$, and such that the orbits passing through $p_1$ and $p_2$ meet tangentially at $p$.
\end{lemma}

The last foundational result of Giroux rules out the appearance of several configurations of dividing sets in tight contact $3$-manifolds:

\begin{theorem}[Giroux Tightness Criteria\cite{GirouxCircleBundles}]\label{teo:Criterio-Giroux}
    Let $\Sigma\subseteq (M,\xi)$ be a convex surface with possibly empty Legendrian boundary. Then, an $I$-invariant neighborhood of $\Sigma$ is tight if and only if one of the following two conditions holds
    \begin{enumerate}[label=(\roman*)]
        \item $\Sigma=\NS^2$ and the dividing set $\Gamma$ is connected, or
        \item $\Sigma\neq \NS^2$ and the dividing set does not contain homotopically trivial curves.
    \end{enumerate}
\end{theorem}

%%%%%%%%%%%%%%%%%%%%%%%%%%%%%%%%%%%%%%%
\subsection{Bypasses in contact $3$-manifolds.} \label{subsec:bypasses}
%%%%%%%%%%%%%%%%%%%%%%%%%%%%%%%%%%%%%%%
A \em bypass \em is a convex half-disk $D$ cobounded by two Legendrian curves $\gamma$ and $\beta$, which we will call \em attaching arc \em and \em Honda arc\em, respectively. The attaching arc has three elliptic singularities of alternating sign and shares two of them with the Honda arc, which also has a hyperbolic singularity. Finally, the dividing set of a bypass consists on a single embedded arc with end points at the attaching arc.  See Figure \ref{fig:bypass1}. 
A bypass can be \em positive \em or \em negative\em, depending on the sign of the central elliptical singularity. In this article, unless other thing is stated, we will assume that every bypass is positive. However, everything will apply for negative bypasses in the same way.
\begin{figure}[h]
    \centering
    \includegraphics[scale=0.3]{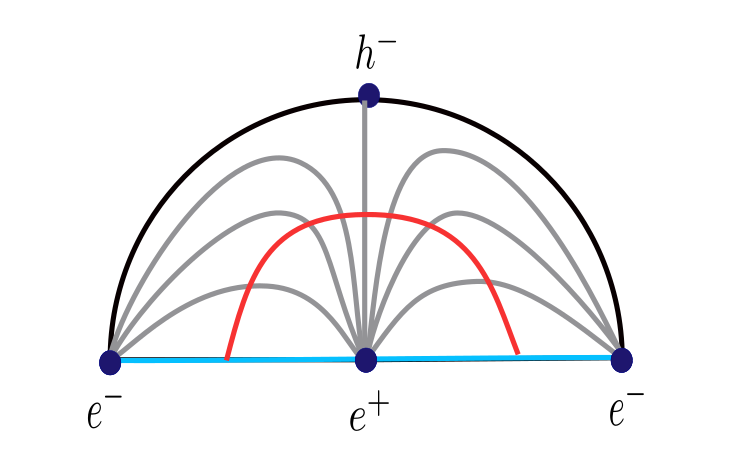}
    \caption{A positive bypass, where we can see in light blue the attaching arc, in black the Honda arc, in red the dividing set and finally in gray the characteristic foliation.}
    \label{fig:bypass1}
\end{figure}

Bypasses are relevant because they describe the failure of convexity in $1$-parametric families of surfaces. More precisely, every smooth isotopy of between two convex surfaces can be assumed to be conformed by a sequences of bypass attachments that we now describe \cite{Colin,HondaGluing,GirouxFoliations}. Given a convex surface $\Sigma$ and a bypass disk $D$ that intersects $\Sigma$ only at the attaching arc $\gamma=\Sigma\cap D$, the neighborhood of $(\Sigma\times I,\xi)=(\overline{\Op(\Sigma\cup D)},\xi)$ is called \em bypass attachment. \em Here, $\Sigma=\Sigma\times\{\varepsilon\}$ for $\varepsilon >0$, $\Sigma\times\{1\}$ is convex and its dividing set can be obtained from the one of $\Sigma$ by following the modification like that in Figure \ref{fig:bypass-2}.
\begin{figure}[h]
    \centering
    \includegraphics[scale=0.3]{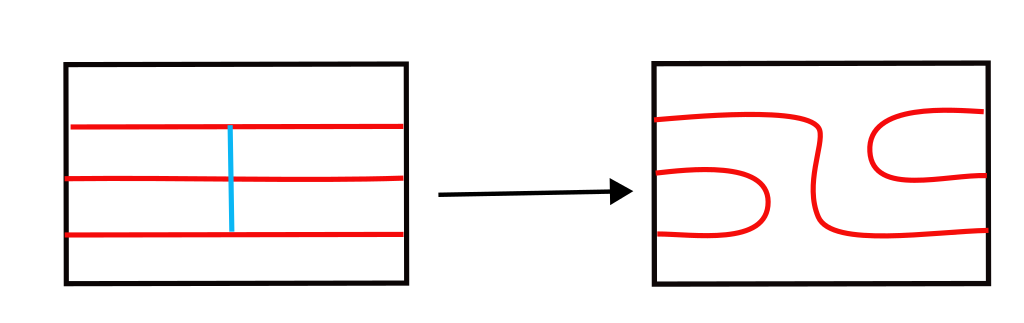}
    \caption{The effect of a bypass attachment.}
    \label{fig:bypass-2}
\end{figure}
Given a Legendrian arc $\gamma\subseteq \Sigma$ intersecting the dividing set of $\Sigma$ at exactly $3$ points we will denote by $\mathcal{B}_\gamma((M,\xi);\Sigma)$ the space of embeddings of bypasses for $\Sigma$ which coincides over a open neighborhood of the attaching arc $\gamma$ and with fixed characteristic foliation. If we write $\mathcal{B}_\gamma(M,\xi)$ we will simply mean the spaces of all bypasses with attaching arc $\gamma$ that coincide near $\gamma$.

Observe that, in general, the space $\mathcal{B}_\gamma((M,\xi);\Sigma)$ could be empty. However, there are two important instances in which this is not the case. The first one, as observed by Vogel \cite{Vogel} (see also \cite{Huang}) is when $(M\backslash \Sigma,\xi)$ is overtwisted: 
\begin{lemma}[\cite{Vogel}]\label{lem:bypassAttachingArc}
    If $M\setminus \Sigma$ is overtwisted and $\gamma\subset \Sigma$ is a possible attaching arc then exists bypass with attaching arc $\gamma$ to $\Sigma$.
\end{lemma}
To build the bypass in this Lemma the only necessary data is an overtwisted disk $\Delta_{\OT}\subseteq (M\backslash \Sigma,\xi)$ and a Legendrian curve $\delta:[0,1]\rightarrow (M,\xi)$ such that $\delta(0)\in \gamma\subseteq \Sigma$, $\delta(1)\in \Delta_{\OT}$ and $\delta(0,1)\cap \Sigma=\delta(0,1)\cap \Delta_{\OT}=\emptyset$. The bypass is given by (a perturbation) of the half-disk cobounded by $\gamma$ and the Legendrian connected sum of $\gamma$ and $\partial \Delta_{\OT}$ following $\beta$.

The second case in which the existence of bypasses is guaranteed is for \em trivial bypasses. \em This is known as \em Right to Life Principle \em in the literature \cite{HondaGluing,HondaKazezMatic}. We say that a bypass is \em trivial \em for $\Sigma$ if the effect of attaching the bypass does not change the isotopy class of the dividing set. A Legendrian arc $\gamma\subseteq \Sigma$ intersecting the dividing set at precisely $3$ points is said to be a \em trivial attaching arc \em if every (abstract) bypass for $\Sigma$ along $\gamma$ is trivial. 
\begin{figure}[h]
        \centering
        \includegraphics[scale=0.3]{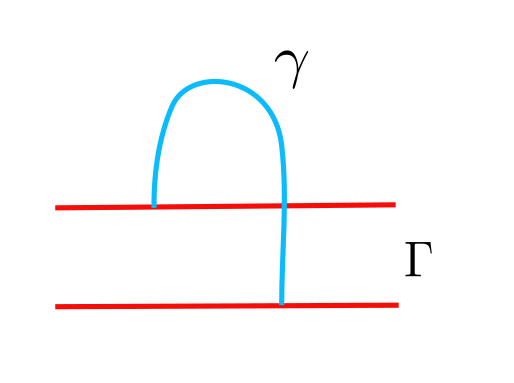}
        \caption{Trivial bypass attachment}
        \label{fig:bypass-trivial}
\end{figure}

\begin{lemma}[Honda]\label{lem:trivial}
    Let $(M,\xi)$ be a contact $3$-manifold, $\Sigma$ be a convex surface with possibly empty Legendrian boundary and $\gamma \subset \Sigma$ be a trivial attaching arc. Then, there exists a trivial bypass $D$ with attaching arc $\gamma$. Moreover, the neighborhood $(\overline{\Op(\Sigma \cup D)},\xi)$ is contactomorphic to the standard $I$-invariant neighborhood $(\Sigma\times I,\xi_{inv})$.
\end{lemma}

%%%%%%%%%%%%%%%%%%%%%%%%%%%%%%%%
\subsection{(De)stabilizations, bypasses and Fuchs-Tabachnikov}\label{subsec:stabilizations}
%%%%%%%%%%%%%%%%%%%%%%%%%%%%%%%%
We introduce Legendrian stabilizations as a Legendrian satellite operation, e.g. \cite{EtnyreVertesi}, so it can be carried out parametrically \cite{FMP24}. 

Consider the contact $3$-manifold $(\NS^1\times \D^2,\xi_{\std})\subseteq (J^1\NS^1,\xi_{\std}=\ker(dy-xd\theta))$,
where $(x,y)\in\D^2$ and $\theta\in \NS^1$; and the \em Lengendrian 0-section \em $s_0:\NS^1\hookrightarrow\NS^1\times\{0\}\subset\NS^1\times\D^2$. Fix a small $\varepsilon>0$, we define the positive and negative stabilizations of $s_0$ at time $t$ as the Legendrian curves $$s^t_\pm:\NS^1\hookrightarrow (\NS^1\times \D^2,\xi_{\std})$$ 
whose fronts are depicted in Figure \ref{fig:stabilizations}. These satisfy that $s^t_{\pm}(z)=s_0(z)$ for all $z\in \NS^1\backslash (t-\varepsilon,t+\varepsilon)$, whereas for $z\in (t-\varepsilon,t+\varepsilon)$ their front traces a zigzag as depicted in the figure. This definition is independent of $\varepsilon>0$ up to Legendrian isotopy. 
\begin{figure}[h]
    \centering
    \includegraphics[scale=0.4]{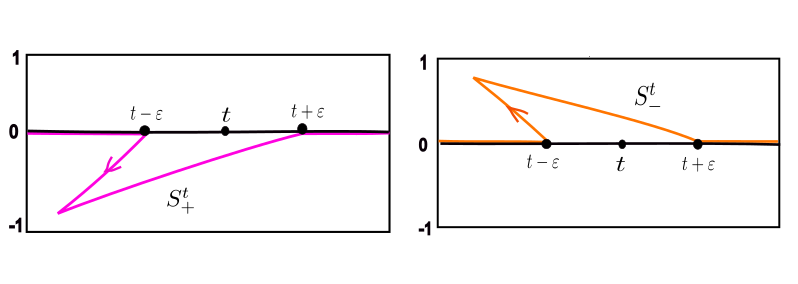}
    \caption{Stabilizations of $s_0$}
    \label{fig:stabilizations}
\end{figure}

Let us recall the Legendrian satellite construction. Consider a Legendrian embedding  $\gamma$ into a contact $3$-manifold $(M^3,\xi)$ together with a standard contact neighborhood of it, $T_\gamma : (\NS^1 \times \D^2, \xi_{std}) \to (M^3, \xi),$ where  $T_\gamma(\theta, 0, 0) = \gamma(\theta)$.
Let $\beta : \NS^1 \to (\NS^1 \times \D^2, \xi_{std})$ be another Legendrian embedding. Associated with this data we define \em Legendrian satellite \em embedding: 
$$T_\gamma \circ \beta : \NS^1 \to (M^3, \xi).$$
The Legendrian $\gamma$ is called \em companion embedding \em and $\beta$ the \em pattern embedding\em. The uniqueness, up to contact isotopy, of the standard neighborhood of $\gamma$ implies that this definition is independent of $T_\gamma$ up to Legendrian isotopy.

Given a Legendrian embedding $\gamma\in \Leg(M,\xi)$ we define the positive, resp. negative, \em stabilization \em of $\gamma$ at time $t\in\NS^1$ as the Legendrian satellite embedding with companion $\gamma$ and pattern the curve $s^t_+:\NS^1\hookrightarrow (\NS^1\times\D^2,\xi_{\std})$, resp. $s^t_-$, defined above.  Following \cite{FMP24} there exist well-defined stabilization maps at time $t\in \NS^1$ 
\begin{equation}\label{eq:StabilizationMaps}
S^t_{+}:\Leg(M,\xi)\to \Leg(M,\xi) \text{ and } S^t_{-}:\Leg(M,\xi)\to \Leg(M,\xi).
\end{equation}

We say that $\beta'$ is a \em destabilization \em of $\beta$ if $S^t_\pm(\beta')=\beta$ for some $t$.
It follows from the local construction at $(\NS^1\times\D^2,\xi_{\std})$ that $s_0$ is a destabilization of both $s^t_{+}$ and $s^t_{-}$. 
Moreover, there is a bypass disk $b_+$, resp. $b_-$, naturally co-bounded by the Honda arc $-s_0[t-\varepsilon,t+\varepsilon]$ and by the attaching arc $s^t_+[t-\varepsilon,t+\varepsilon]$, resp.$s^t_-[t-\varepsilon,t+\varepsilon]$. The bypass $b_+$ is positive whereas the $b^-$ is negative. We depicted the bypass $b_+$ in Figure \ref{fig:bypass3}.
\begin{figure}[h]
    \centering
    \includegraphics[scale=0.4]{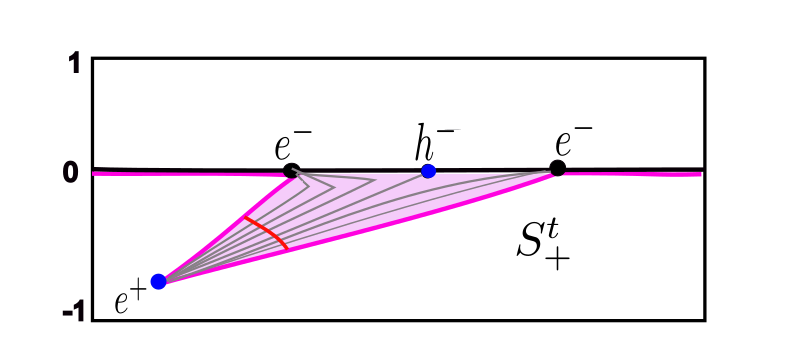}
    \caption{Bypass disk resulting from a stabilization}
    \label{fig:bypass3}
\end{figure}

In general, given $\beta$ a Legendrian embedding with $\gamma=\beta_{|I_0}$, 
where $I_0\subset\NS^1$ an subinterval and $b:D\to(M,\xi)$ a bypass disk embedding such that the attaching arc
of $b$ coincides with $\gamma$ and $b(D)\cap\beta (\NS^1)=\gamma(I_0)$, we can find a new Legendrian embedding $\beta'$
such that $S^t_\pm(\beta')=\beta$, i.e. $\beta'$ is a destabilization of $\beta$.
To do this we replace $\gamma$ by the Honda arc of $b$ with the opposite orientation after using the pivot Lemma \ref{lem:pivot}
twice to smooth the corners. Since Giroux realization is parametric all this process can be done parametrically meaning that given a family of Legendrian curves and bypasses as before one can simultaneously destabilize the Legendrians.

The main motivation behind this work rests on the previous destabilization procedure by means of bypasses embeddings together with the Fuchs-Tabachnikov Theorem \cite{FuchsTabachnikov} that states that two formally isotopic Legendrians became Legendrian isotopic after adding sufficiently many stabilizations (of both signs). In this article, we will make use of the parametric version of this result proved by Casals-del Pino \cite{CasalsDelPino}, see also Murphy \cite{Murphy}.  
    
\begin{theorem}[Fuchs-Tabachnikov \cite{CasalsDelPino}]\label{teo:Fuchs-Taba}
    Let $(M,\xi)$ be a contact 3-manifold and $\gamma^z: \NS^1\to(M ,\xi)$ a parameterized family of formal Legendrian knots 
    $z\in \D^k$, such that $\gamma^z$ Legendrian knots for $z \in \mathcal{O}p(\partial \D^k)$. Then, there exists a 
    parameterized family $\tilde{\gamma}^z$, $z\in \D^k$, of Legendrian knots such that $\tilde{\gamma}^z$, $z \in \mathcal{O}p(\partial\D^k)$, 
    is obtained from $\gamma^z$, $z\in \Op(\partial\D^k)$, by 
 adding pairs of positive and negative stabilizations several times.
\end{theorem} 

\begin{remark}
The previous result states that multiple applications of the operators $S^t_+$ and $S^t_-$, defined in (\ref{eq:StabilizationMaps}), on the space of Legendrian embeddings allow for gaining flexibility. This is analogous to the situation in higher dimensions, with the crucial difference being that in dimension 3, the stabilization map changes the Legendrian isotopy class, whereas, as observed by Murphy \cite{Murphy}, this is not the case in higher dimensions for stabilizations along circles, as defined in \cite{EkholmEtnyreSullivan}, for Legendrians with at least one stabilization, i.e., Loose Legendrians, since these can be self-replicated. Nevertheless, one can still attempt to use this result to prove flexibility results for Legendrian embeddings in dimension 3 by trying to invert (in homotopy) the operators $S^t_+$ and $S^t_-$. This, of course, is not generally possible \cite{Chekanov}. However, it follows from our previous discussion that the obstruction to doing this is measured by the homotopy groups of a certain space of bypass embeddings. Therefore, Theorem \ref{teo:bypassMiddle} can be applied to show that this indeed can be achieved in some special situations (see Theorem \ref{thm:StabilizationMapsLooseLegendrians}).
\end{remark}

%%%%%%%%%%%%%%%%%%%%%%%%%%%%%
\subsection{Contact Hamiltonians}\label{subsec:contactHamiltonian}
%%%%%%%%%%%%%%%%%%%%%%%%%%%%%
Let $(M,\xi)$ be a contact manifold. Fix a contact form $\alpha\in \Omega^1(M)$; i.e. $\xi=\ker(\alpha)$, so that there is an associated \em Reeb field \em $R$ characterized by the equations $\alpha(R)=1$ and $i_{R}d\alpha=0$. 

The tangent space of the contactomorphism group $\Cont(M,\xi)$ at the Identity, i.e. its Lie Algebra, is precisely the space of contact vector fields. The later can be identified with the space of smooth function $C^\infty(M)$  by assigning to each contact vector field $X$ the smooth function $H_X:=\alpha(X)$ known as \em contact Hamiltonian. \em This assignment identifies the space of contact vector fields with the space of smooth functions $C^\infty(M)$. Indeed, given a function $H\in\C^\infty(M)$ there is a unique contact vector field $X_H$ determined by the relations $\alpha(X_H)=H$ and $i_{X_H} d\alpha = dH(R )\alpha - dH$.

%%%%%%%%%%%%%%%%%%%%%%%%%%
\section{The Bypass in the Middle Theorem}\label{sec:BypassInTheMiddle}
%%%%%%%%%%%%%%%%%%%%%%%%%%%

In this section we will prove the Bypass in the Middle Theorem \ref{teo:bypassMiddle}.
The main ingredient in the proof is Theorem \ref{thm:TrivialBypasses}, that could be thought as a parametric version of Honda's triviality Lemma \ref{lem:trivial}.

%%%%%%%%%%%%%%%%%%%%%%%%%%%%%%%%%%%%%%%%%%
\subsection{The space of trivial bypasses}
%%%%%%%%%%%%%%%%%%%%%%%%%%%%%%%%%%%%%%%%%
\begin{theorem}\label{thm:TrivialBypasses}
   Let $\Sigma \subset(M^3,\xi)$ be a convex surface with possibly empty Legendrian boundary and $\gamma$ attaching arc for $\Sigma$ of a trivial bypass. Then, the space $\Bypass_\gamma((M,\xi),\Sigma)$ is contractible.
\end{theorem}
\begin{proof}
Let $b\in\Bypass_\gamma((M,\xi),\Sigma)$ and $D=\image(b)$. The map $$\Cont_0(M,\xi;\rel \Sigma)\rightarrow \Bypass_\gamma((M,\xi),\Sigma), \varphi\mapsto \varphi\circ b;$$ defines a fibration
$$\Cont_0(M,\xi, \rel \Sigma\cup D)\hookrightarrow \Cont_0(M,\xi, \rel \Sigma)\to\Bypass_\gamma((M,\xi),\Sigma). $$
Here $\Cont_0(M,\xi, \rel \Sigma)$  stands for the path connected components of the identity map and $\Cont_0(M,\xi,\rel \Sigma \cup D)=\Cont_0(M,\xi;\rel \Sigma)\cap \Cont(M,\xi,\rel \Sigma\cup D)$. It follows from Lemma \ref{lem:surjective}, proved below, that the fibration map is surjective. Moreover, by Lemma \ref{lem:fiber}, also proved below; the inclusion of the fiber in the total space is a weak homotopy equivalence. Therefore, the base is contractible and the result follows.
\end{proof}

It remains to prove Lemmas \ref{lem:surjective} and \ref{lem:fiber}.

\begin{lemma}\label{lem:surjective}
    The map $\Cont_0(M,\xi, \rel \Sigma)\to \Bypass_\gamma((M,\xi),\Sigma),\varphi\mapsto \varphi\circ b;$ is surjective.
\end{lemma}

\begin{proof}
    Let $\tilde{b}\in\Bypass_\gamma((M,\xi),\Sigma)$ be a trivial bypass. We must find a contactomorphism $\tilde{\varphi}$, contact isotopic to the identity such that $\tilde{\varphi}\circ b=\tilde{b}$.
    
    Write $D=\image(b)$ and $\tilde{D}=\image(\tilde{b})$ and $I=[0,1]$. Consider an I-invariant neighborhood $(\Sigma\times I, \xi_{inv}=\ker(fdt+\beta))$ of $\Sigma$, where $f\in C^\infty(\Sigma)$ and $\beta\in \Omega^1(\Sigma)$. Since $b$ and $\tilde{b}$ are trivial bypasses, Lemma \ref{lem:trivial} implies the existence of contactomorphisms 
    $$F:(\Sigma\times I,\xi_{inv})\to (\overline{\Op(\Sigma\cup D)},\xi),$$
    $$\tilde{F}:(\Sigma\times I,\xi_{inv})\to (\overline{\Op(\Sigma\cup \tilde{D}}),\xi).$$
    We can assume that 
    \begin{enumerate}[label=(\roman*)]
        \item $F_{|\Sigma\times[0,2\varepsilon]}=\tilde{F}_{|\Sigma\times[0,2\varepsilon]}$ for some small $\varepsilon>0$,
        \item $F(\Sigma\times \{\varepsilon\})=\tilde{F}(\Sigma\times \{\varepsilon\})=\Sigma\subseteq M$.
    \end{enumerate}
    Fix $\delta\in(0,\varepsilon)$ small such that $\tilde{D}\subset \tilde{F}(\Sigma\times [0,1-\delta])$. Consider the Hamiltonian $$H(p,t)=-\rho(t)f(p),$$ in $\Sigma\times I$,
    where $\rho:[0,1]\to[0,1]$ is a smooth bump function such that 
    \begin{enumerate}[label=(\alph*)]
        \item $\rho(t)=0$ for $t\in \Op([0,\varepsilon]\cup \{1\})$ and 
        \item $\rho(t)=1$ for $t \in \Op([\varepsilon+\delta,1-\delta])$.
    \end{enumerate}
    Let $X_H$ be the associated contact vector field in $(\Sigma\times I,\xi_{inv})$ with respect to the contact form $\alpha=f dt+ \beta$. Then, the contact vector field $\tilde{F}_*X_H$ generates a contact flow \break $g_u\in \Cont_0(M,\xi,\rel\Sigma)$, $u\in \R$,
    with support in $\tilde{F}(\Sigma\times I)$.  Fix $U>0$ big enough so that $g_U(\tilde{D})\subset F(\Sigma\times I)$ and consider 
    $\hat{b}:=g_U\circ\tilde{b}$. 

    Let $\hat{\gamma}\times\{\varepsilon\}=F^{-1}(\gamma)\subseteq \Sigma\times\{\varepsilon\}$ be the pre-image under $F$ of the attaching arc of both $D$ and $\hat{D}:=\image(\hat{b})=g_U(\tilde{D})$. Again Lemma \ref{lem:trivial} implies that
    $$(\Sigma\times I\setminus\Op(F^{-1}(D)\cup \hat{\gamma}\times[0,\varepsilon])),\xi_{inv})\cong (\Sigma\times I\setminus\Op(F^{-1}(\hat{D})\cup \hat{\gamma}\times[0,\varepsilon]),\xi_{inv})\cong(\Sigma\times I,\xi_{inv}).$$ This, together with Theorem \ref{teo:giroux-entorno}, allow us to find
    $\hat{\varphi}\in \Cont(\Sigma\times I,\xi_{inv})$ such that $\hat{\varphi}\circ F^{-1}\circ b=F^{-1} \circ \hat{b}$. We claim that we can assume that $\hat{\varphi}$ is contact isotopic to the identity. Indeed, if this is not the case we can use a contact Hamiltonian as before to isotope $\hat{\varphi}^{-1}$ to a contactomorphism $G\in \Cont(\Sigma\times I,\xi_{inv})$ with support in $\Op(\Sigma\times\{1\})$. Hence, $G\circ \hat{\varphi}\in \Cont_0(\Sigma\times I,\xi_{inv})$ would be contact isotopic to the identity and still satisfy $G\circ \hat{\varphi}\circ F^{-1}\circ b=F^{-1}\circ \hat{b}$. 
    
    Therefore, the contactomorphism $$\tilde{\varphi}=g_U^{-1}\circ F\circ \hat{\varphi}\circ F^{-1}\in \Cont_0(M,\xi,\rel \Sigma)$$ is contact isotopic to the identity, here $F\circ \hat{\varphi}\circ F^{-1}$ is defined to be the identity outside $\image(F)$, and satisfies that $ \tilde{\varphi} \circ b= \tilde{b}$ as required. This concludes the argument.  
\end{proof}

\begin{lemma}\label{lem:fiber}
    The inclusion of the fiber 
    $$\Cont_0(M,\xi, \rel \Sigma\cup D)\hookrightarrow \Cont_0(M,\xi, \rel \Sigma)$$ is a weak homotopy equivalence.
\end{lemma}

\begin{proof} 
    Let $(K,G)$ be a compact CW-pair and $$\Phi:(K,G)\rightarrow (\Cont_0(M,\xi,\rel \Sigma),\Cont_0(M,\xi,\rel \Sigma\cup D))$$ a continuous map. It is enough to find a homotopy $$\Phi_s:(K,G)\rightarrow (\Cont_0(M,\xi,\rel \Sigma),\Cont_0(M,\xi,\rel \Sigma\cup D)), s\in[0,1],$$ such that 
    \begin{enumerate}[label=(\alph*)]
        \item $\Phi_0=\Phi,$
        \item $\Phi_s(g)=\Phi(g)$, for $(s,g)\in [0,1]\times G$, 
        \item $\image(\Phi_1)\subseteq \Cont_0(M,\xi,\rel \Sigma\cup D)$.
   \end{enumerate}

    Lemma \ref{lem:trivial} implies the existence of a contactomorphism
    $$F:(\Sigma\times I,\xi_{inv})\to (\overline{\Op(\Sigma\cup D)},\xi)),$$
    where $\xi_{inv}=\ker(fdt+\beta)$ and $F(\Sigma\times\{\varepsilon\})=\Sigma\subseteq M$ for some $\varepsilon>0$ small. We may further assume that $\Phi(k)_{|\image(F)}=\Id$ for all $k\in G$.
    
    Fix $\delta>0$ small enough such that 
    \begin{enumerate}[label=(\roman*)]
        \item $D\subseteq  F(\Sigma\times [0,1-\delta])$,
        \item $\Phi(k)_{|\Op(F(\Sigma\times [0,\varepsilon+\delta])}=\Id$ for all $k\in K$.
    \end{enumerate}

    Let $X_H$ be the contact vector field in $(\Sigma\times I,\xi_{inv})$ defined in the proof of Lemma \ref{lem:surjective}. Consider the contact flow $g_u\in \Cont_0(M,\xi,\rel\Sigma)$, $u\in\R$, with support in $F(\Sigma\times[0,1])$, of the contact vector field $F_*X_H$. Fix $U>0$ big enough such that $g_U(F(\Sigma\times [0,1-\delta]))$ is contained in $F(\Sigma\times [0,\varepsilon+\delta])$. The required homotopy is given by 
    $$\Phi_s= g_{sU}^{-1}\circ \Phi \circ g_{sU}, s\in[0,1]. $$
\end{proof}

%%%%%%%%%%%%%%%%%%%%%%%%%%%%%%%%%
\subsection{Proof of Theorem \ref{teo:bypassMiddle}}
%%%%%%%%%%%%%%%%%%%%%%%%%%%%%%%%%%

Let us explain the intuition behind the proof of Theorem \ref{teo:bypassMiddle}, where the goal is to show that a family of bypasses admitting a bypass in the middle is contractible. The existence of a bypass in the middle allows us to reduce the problem to a parametric family of tight $3$-balls, where the problem can be solved in each ball individually (Theorem \ref{thm:TrivialBypasses}). To piece together the different solutions corresponding to the various balls, we use a microfibration argument, where the contractibility of the fibers is ensured once more by Theorem \ref{thm:TrivialBypasses}.

\begin{proof}[Proof of Theorem \ref{teo:bypassMiddle}]Let $K$ be a compact parameter space and $b^z\in \mathcal{B}_\gamma^b(\Sigma,(M,\xi))$, $z\in K$, a continuous family of bypasses, all of them disjoint from $b$ away from the attaching arc $\gamma$. 

It is enough to find a homotopy $b^{z,t}\in\mathcal{B}_\gamma(\Sigma,(M,\xi))$, $t\in [0,1]$, such that 
\begin{enumerate}[label=(\alph*)]
    \item $b^{z,0}=b^z$,
    \item $b^{z,1}=b$.
\end{enumerate}

We will reformulate the existence of such a homotopy as the existence of a section of a (micro)fibration with non-empty contractible fiber so the result will follow.

Write $D^z=\image(b^z)$, $z\in K$, and $D=\image(b)$. For each $z\in K$, the half-disks $D^z$ and $D$ are disjoint away from $\gamma$; hence, we can find an embedded closed $3$-disk $C^z$ such that
\begin{enumerate}[label=(\roman*)]
    \item $D^z \cup D\subseteq C^z$,
    \item $\Sigma\cap C^z=\Sigma\cap \partial C^z=\overline{\mathcal{O}p_\Sigma(\gamma)}$ is independent of $z\in K$, and  
    \item $(C^z,\xi)$ is tight.
\end{enumerate}
 Conditions (i) and (ii) are obvious. To check condition (iii) observe that, as a consequence of Theorem \ref{teo:giroux-entorno}, $(\Op(D^z\cup D),\xi)$ contactly embeds in an $I$-invariant neighborhood of $D$, which is tight by Giroux Criteria Theorem \ref{teo:Criterio-Giroux}. Therefore, taking $C^z$ small enough we can ensure tightness.

 \begin{figure}[h]
    \centering
    \includegraphics[scale=0.3]{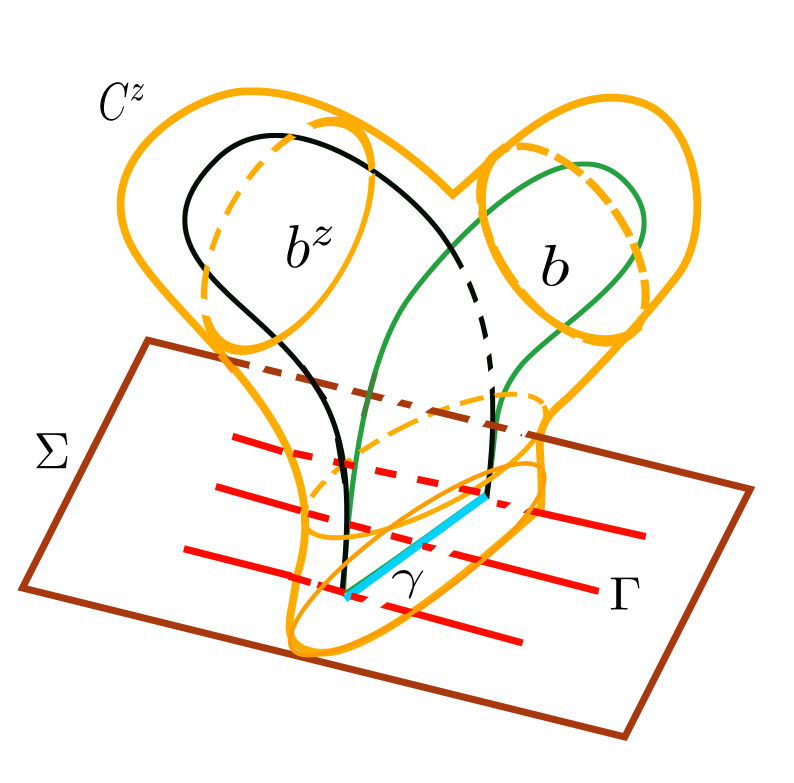}
    \caption{Schematic depiction of the $3$-disk $C^z$.}
    \label{fig:middle1}
\end{figure}

Denote by $\mathcal{X}$ the space of pairs $(z,\hat{b}^{z,t})$ where 
\begin{enumerate}[label=(\alph*)]
    \item $z\in K$,
    \item $\hat{b}^{z,t}\in\mathcal{B}_\gamma(\partial C^z,(C^z,\xi))$, $t\in [0,1]$, is a homotopy of bypasses between $b^z=\hat{b}^{z,0}$ and $b=\hat{b}^{z,1}$.
\end{enumerate}
The projection $$\pi:\mathcal{X}\rightarrow K,\quad (z,\hat{b}^{z,t})\mapsto z;$$ is naturally a microfibration. Moreover, since each $(C^z,\xi)$ is tight, every bypass attached to $\partial C^z$ inside $C^z$ is trivial. Therefore, every fiber is contractible because of Theorem \ref{thm:TrivialBypasses}. By Weiss Microfibration Lemma \cite{Weiss} we conclude that
$\pi$ is a Serre fibration with contractible fiber. \footnote{In fact, choosing $C^z$ carefully enough one can ensure that it is a Serre fibration, without appealing to the Microfibration Lemma.} In particular, there exists a section $s:K\rightarrow \mathcal{X}$. The second component of $s(k)$ is the desired homotopy.
\end{proof}

\begin{remark}
    The previous argument can be easily to prove the following more general property. Let $K$ be a compact parameter space and $B_0,B_1:K\rightarrow \mathcal{B}_\gamma(\Sigma,(M,\xi))$ two families of bypass embeddings such that 
    $$ B_0(k)\in \mathcal{B}_\gamma^{B_1(k)}(\Sigma,(M,\xi)) $$
    for all $k\in K$. Then, $B_0$ and $B_1$ are homotopic through families of bypass embeddings. The details are left to the reader.
\end{remark}

%%%%%%%%%%%%%%%%%%%%%%%%%%%%%%%%%%%
\section{Applications to overtwisted contact $3$-manifolds}\label{sec:Applications}
%%%%%%%%%%%%%%%%%%%%%%%%%%%%%%%%%
In this Section we will prove all the corollaries of the Bypass in the Middle Theorem \ref{teo:bypassMiddle} stated in \ref{subsec:Applications}.

%%%%%%%%%%%%%%%%%%%%%%%%%%%%%%%%%%%%%%%%%%%%%%%%%%%%%%%%%%%%
\subsection{Proof of Corollary \ref{coro:middle-OT}}
%%%%%%%%%%%%%%%%%%%%%%%%%%%%%%%%%%%%%%%%%%%%%%%%%%%%%%%%%%%%%

Let $b^z\in \Bypass_\gamma(\Sigma, (M\backslash \Delta_{\OT},\xi))$, $z\in \NS^k$, be a family of bypasses.
We explain how to homotope this family to a constant bypass. 

Since $b^z$ is contractible, relative to $\gamma$, as a family of smooth half-disk embeddings, we can find a family of smooth arcs $\delta^{z,r}:[0,1]\rightarrow M$, $(z,r)\in \NS^k\times[0,1]$, such that 
\begin{itemize}
    \item $\delta:=\delta^{z,0}$ is independent of $z\in \NS^k$. 
    \item $\delta^{z,r}_{|\Op(\{0,1\})}=\delta_{|\Op(\{0,1\})}$ are Legendrian arcs, for all $(z,r)\in \NS^k\times [0,1]$. 
    \item $\image(\delta^{z,r})\cap \Delta_{\OT}=\delta^{z,r}(1)=\delta (1)$.
    \item $\image(\delta^{z,r})\cap \Sigma=\delta^{z,r}(0).$
    \item $\delta^{z,1}(t)=b^z(0,t)$ for $t\in[0,\varepsilon]$, where $\varepsilon>0$ is a small constant. 
    \item $\image(\delta^{z,1}_{|(\varepsilon, 1]})\cap \image(b^z)=\emptyset$.
\end{itemize}

The first condition implies that the parameter space can be thought as a disk. Therefore, we can apply Fuchs-Tabachnikov Theorem \ref{teo:Fuchs-Taba} to further ensure that 
\begin{itemize}
    \item $\delta^{z,r}$ is Legendrian.
\end{itemize}

Apply the Contact Isotopy Extension Theorem to the family of Legendrian arcs $\delta^{z,r}$ to find a family of contactomorphisms $\varphi^{z,r}\in \Cont(M,\xi;\rel \Sigma\cup \Delta_{\OT})$, $(z,r)\in \NS^k\times[0,1]$, such that 
\begin{itemize}
    \item $\varphi^{z,0}=\Id$,
    \item $\varphi^{z,r}\circ \delta=\delta^{z,r}$.
\end{itemize}

We use this family of contactomorphisms to define a homotopy of bypasses, 
$$ b^{z,r}=(\varphi^{z,r})^{-1}\circ b^z; $$
between $b^{z,0}=b^z$ and a new family $b^{z,1}$. Observe that the Legendrian arc $\delta$ is disjoint from all the bypasses $b^{z,1}$, $z\in\NS^k$, away from the attaching region. Therefore, we can use the arc $\delta$ together with Lemma \ref{lem:bypassAttachingArc} to find a bypass disk embedding $B$ disjoint from the whole family $b^{z,1}$. Apply the Bypass in the Middle Theorem \ref{teo:bypassMiddle} to conclude that $b^{z,1}$ is a contractible family of bypasses. 
\qed

%%%%%%%%%%%%%%%%%%%%%%%%%%%%%%%%%%%%%%%%%%%%%%%%%%%%%%%%%%%%%

\subsection{Proof of Corollary \ref{cor:OvertwistedDisks}}

%%%%%%%%%%%%%%%%%%%%%%%%%%%%%%%%%%%%%%%%%%%%%%%%%%%%%%%%%%%%%

The map $ev_0$ is a surjective Serre fibration, we must show that the fiber is weakly contractible.

Fix a base is point $(p,v)\in\CFr(M\setminus \Delta_{\OT},\xi)$ the fiber is $\Emb^{(p,v)}_{\OT}(\D^2,(M\setminus\Delta_{\OT},\xi))$, which is the space of embeddings $e\in \Emb_{\OT}(\D^2,(M\setminus\Delta_{\OT},\xi))$ such that $ev_0(e)=(e(0),d_0e(e_1,e_2))=(p,v)$.

    By Giroux realization Theorem \ref{teo:realizacion-Giroux} the space $\Emb^{(p,v)}_{\OT}(\D^2,(M\setminus\Delta_{\OT},\xi))$ is weakly homotopy equivalent to the space of convex embeddings $e:\D^2\subset \R^2(x,y)\to (M\setminus\Delta_{\OT},\xi)$
    such that 
    \begin{enumerate}[label=(\roman*)]
        \item $ev_0(e)=(p,v)$,
        \item $e_{|\D^2\cap\{x\leq 0\}}$ and $e_{|\D^2\cap\{x\geq 0\}}$ are bypass disk embeddings.
    \end{enumerate}
    Denote this space as $\DB_{(p,v)}(M\setminus\Delta_{\OT},\xi)$.

    We will show that $\DB_{(p,v)}(M\setminus\Delta_{\OT},\xi)$ is contractible.
    Let $E:\NS^k\to \DB_{(p,v)}(M\setminus\Delta_{\OT},\xi)$ be any sphere. Denote the Legendrian embeddings given by the restriction of each $E(z)$ to the $Y$-axis by
    $\gamma^z=E(z)_{|\D^2\cap\{x=0\}}:[-1,1]\to(M\setminus\Delta_{\OT},\xi),y\mapsto \gamma^z(y)=E(z)(0,y)$.
    
    Condition (i) above implies that, after a possible contact isotopy, we can assume that 
    \begin{itemize}
        \item [(a)] $\gamma^z\equiv \gamma$ is constant and,
        \item [(b)] $E(z)_{|\D^2\cap\{-\varepsilon\leq x\leq\varepsilon\}}$ is constant for some small $\varepsilon>0$.
    \end{itemize}

    Our goal is to apply the Bypass in the Middle Theorem \ref{teo:bypassMiddle} to both families $E(z)_{|\D^2\cap\{x\leq 0\}}$ and $E(z)_{|\D^2\cap\{x\geq 0\}}$ of bypasses to conclude. 

    First, we observe that both families of bypasses are attached to the same convex rectangle with Legendrian boundary. This rectangle is obtained by slightly enlarging $\gamma$ and pushing it forward and backwards in the direction of the contact vector field transverse to $E(z)$. We denote such a rectangle embedding by $R:[-1,1]\times[-2,2]\to (M\setminus\Delta_{\OT},\xi)$. It satisfies the following properties such that
    \begin{itemize}
        \item $\image (R)\cap \image (E(z))=\image (\gamma)$ and $R(0,y)=\gamma(y)$ for $-1\leq y \leq 1$,
        \item For each $z_0\in [-1,1]$, the horizontal segment $H_{z_0}=\{R(z_0,y):-2\leq y \leq 2\}$ is a Legendrian ruling.
        \item For each $y_0\in\{-2,-\frac12,\frac12,2\}$, the vertical segment $V_{y_0}=\{R(z,y_0):-1\leq z \leq 1\}$ is a Legendrian divide.
        \item The dividing set $\Gamma_R$ of $R$ is given by the vertical segments \hfill\break
        $V_{y_0}=\{R(z,y_0):-1\leq z\leq 1\}$, where $y_0\in \{-1,0,1\}$.
    \end{itemize}

    Observe that $E(z)_{|\D^2\cap\{x\geq 0\}}$, $z\in\NS^k$, is a family of bypasses for $R$ with attaching arc $\gamma$ attached from above. On the other hand, $E(z)_{|\D^2\cap\{x\leq 0\}}$, $z\in \NS^k$, is also a family of bypasses for $R$ with attaching arc $\gamma$ but attached from below.

    We apply Corollary \ref{coro:middle-OT} to the family of bypass embeddings $E(z)_{|\D^2\cap\{x\geq 0\}}$ combined with the Contact Isotopy Extension Theorem to find a homotopy $E^r:\NS^k\rightarrow \DB_{(p,v)}(M\setminus\Delta_{\OT},\xi)$, $r\in[0,1]$, so that $E^0=E$ and
    $$ E^1(z)_{|\D^2\cap\{x\geq 0\}}=:\tilde{b} $$
    is a constant bypass. 

    It remains to fix the family $E^1(z)_{|\D^2\cap\{x\leq 0\}}$, $z\in \NS^k$. The argument is completely analogous to the one just explained since the process can be done relative to $\tilde{b}$. This concludes the proof. \qed
    \begin{figure}[h]
        \centering
        \includegraphics[scale=0.4]{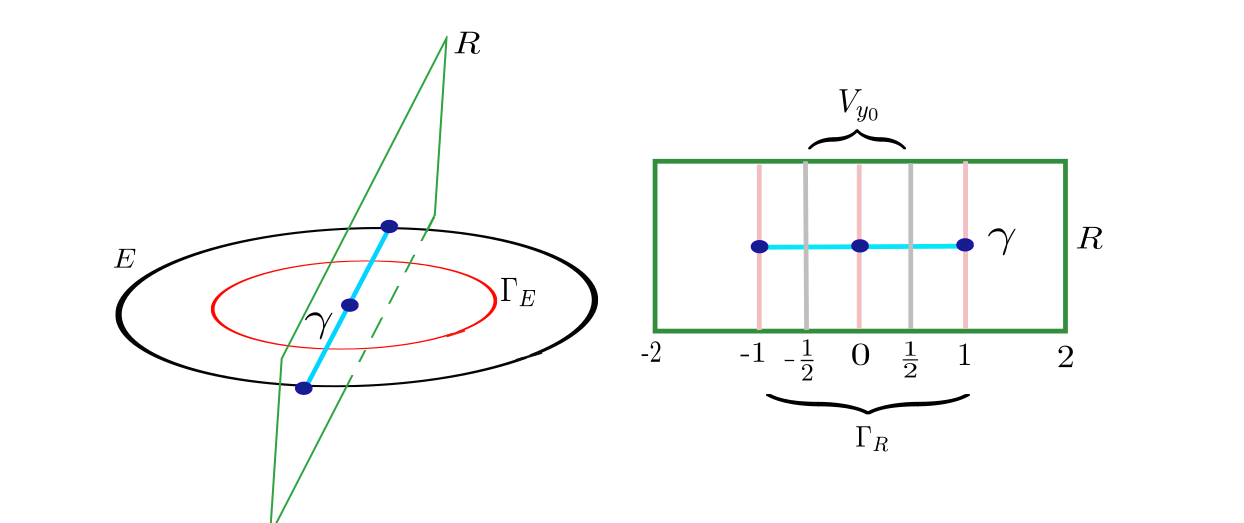}
        \caption{ Bypasses $E(z)_{|\D^2\cap\{x\geq 0\}}$ and $E(z)_{|\D^2\cap\{x\leq 0\}}$ with attaching arc $\gamma$.}
        \label{fig:middle2}
    \end{figure}

%%%%%%%%%%%%%%%%%%%%%%%%%%%%%%%%%%%%%%%%%%%%%%%%%%%%%%%%%%%%%
\subsection{Proof of Theorem \ref{thm:StabilizationMapsLooseLegendrians}}\label{subsec:LooseLegendriansStab}
%%%%%%%%%%%%%%%%%%%%%%%%%%%%%%%%%%%%%%%%%%%%%%%%%%%%%%%%%%%%%

    We prove the result for the positive stabilization map $S_+^t$ since for the negative stabilization map $S_-^t$ the proof is the same. 
  
    The surjectivity of the map  $\pi_0(S^t_+)$ follows from Lemma \ref{lem:bypassAttachingArc}. Indeed, given $\gamma\in \Leg(M\setminus \Delta_{\OT},\xi)$ we can use the Lemma to find a bypass attached to $\gamma$ to destabilize $\gamma$. This is just the well-known fact that the connected sum of a Legendrian and the boundary of an overtwisted disk disjoint form it is a destabilization of the initial Legendrian. 
    
    To conclude we will check that for all $k\geq 1$  and every pair of maps \break $\Phi:\D^k\to \Leg(M\setminus \Delta_{\OT})$ and  $\partial\tilde{\Phi}:\partial\D^k\to \Leg(M\setminus \Delta_{\OT}),$ 
    such that 
    $ \Phi_{|\partial \D^k}=S^t_+\circ\partial\tilde{\Phi};$
    there exists, after a possible homotopy of $\Phi$ relative to $\partial \D^k$, a map $\tilde{\Phi}:\D^k\to \Leg(M\setminus\Delta_{\OT})$ such that
\begin{itemize}
    \item [(i)] $\tilde{\Phi}_{|\partial\D^k}=\partial\tilde{\Phi},$
    \item [(ii)] $S^t_+\circ\tilde{\Phi}=\Phi.$
\end{itemize}

The assumption $\Phi_{|\partial \D^k}=S^t_+\circ\partial\tilde{\Phi}$ is equivalent to the existence of a family of bypass embeddings $B:\partial\D^k\to \mathcal{B}(M\setminus\Delta_{\OT},\xi)$ such that
    \begin{enumerate}
        \item The attaching arc of $B(z)$ lies in $\Phi(z)$ for all $z\in \partial \D^k$
        \item The bypass $B(z)$ and the Legendrian $\Phi(z)$ intersectect only at the attaching arc of $B(z)$ for all $z\in \partial \D^k$.
        \item Replacing the attaching arc of $B(z)$ in $\Phi(z)$ by the Honda arc of $B(z)$, gives (after smoothing with the Lemma \ref{lem:pivot}) the Legendrian $\tilde{\Phi}(z)$.
    \end{enumerate}

Moreover, the existence of the map $\tilde{\Phi}$ will follow from the existence of an extension \break $\tilde{B}:\D^k\to \mathcal{B}(M\setminus\Delta_{\OT},\xi)$ of the map $B$, i.e. $\tilde{B}_{|\partial \D^k}=B$; that satisfies properties (1) and (2) above for all $z\in \D^k$.

Fix coordinates $(z,r)\in \D^k=\partial \D^k\times [0,1]/\simeq $ in the disk $\D^k$. The Contact Isotopy Extension Theorem allows to find a family of contactomorphisms $\varphi:\D^k\to \Cont(M\setminus\Delta_{\OT},\xi)$ such that
    \begin{itemize}
        \item $\varphi(z,0)=\Id$,
        \item $\varphi(z,r)\circ\Phi(z,0)=\Phi(z,r)$ for $(z,r)\in\partial\D^k\times[0,1]$. 
    \end{itemize}
Define the family of bypasses $\hat{B}:\partial\D^k\times[0,1]\to  \mathcal{B}(M\setminus\Delta_{\OT},\xi)$ via the expression $$\hat{B}(z,r)=\varphi(z,r)\circ\varphi^{-1}(z,1)\circ B(z,1).$$ This family extends the initial family of bypasses $B$, i.e. $\hat{B}_{|\partial\D^k\times\{1\}}=B$; and satisfies properties (1) and (2) above. However, in general, the map $\hat{B}_{|\partial \D^k\times\{0\}}$ is not constant as required and defines a $\partial \D^k$-family of bypasses, with fixed attaching arc over the Legendrian $\Phi(z,0)$. Nevertheless, since all these bypasses avoid a fixed overtiwsted disk, we can invoke Corollary \ref{coro:middle-OT} to contract this family to a fixed bypass and, therefore, define the required map $\tilde{B}:\D^k\to  \mathcal{B}(M\setminus\Delta_{\OT},\xi)$. This concludes the argument. 
\qed

%%%%%%%%%%%%%%%%%%%%%%%%%%%%%%%%%%%%%%%%%%%%%%%%%%%%%%%%%%%%%

\subsection{Proof of Corollary \ref{cor:LooseLegendrians}}

%%%%%%%%%%%%%%%%%%%%%%%%%%%%%%%%%%%%%%%%%%%%%%%%%%%%%%%%%%%%%

This essentially follows from Theorem \ref{thm:StabilizationMapsLooseLegendrians} combined with the Fuchs-Tabachnikov theorem \ref{teo:Fuchs-Taba}. We must show that the maps 
$$ \pi_k(i):\pi_k(\Leg(M\backslash \Delta_{\OT},\xi))\rightarrow \pi_k(\FLeg(M\backslash \Delta_{\OT}))$$ induced by the natural inclusion $i:\Leg(M\backslash \Delta_{\OT},\xi)\hookrightarrow \FLeg(M\backslash \Delta_{\OT},\xi)$ are injective and surjective for all $k$. 

We check first the injectivity property. Let $A\in\ker(\pi_k(i))$ be an element of the kernel. Then, Theorem \ref{teo:Fuchs-Taba} implies that there exists times $t_{1}^{+},t_{1}^{-},t_{2}^{+},t_{2}^{-},\ldots,t_{n}^{+},t_{n}^{-}\in\NS^1$ so that 
$$ A \in \ker \pi_k(S^{t_{n}^{-}}_-\circ S^{t_{n}^{+}}_+\circ \cdots S^{t_{1}^{-}}_-\circ S^{t_{1}^{+}}_+).$$
By Theorem \ref{thm:StabilizationMapsLooseLegendrians} we conclude $A=0$.

Let us prove that the map is surjective. Let $A\in\pi_k(\FLeg(M\backslash \Delta_{\OT},\xi))$ be a formal class. Then, by Theorem \ref{teo:Fuchs-Taba} there exists another formal class $\hat{A}\in\pi_k(\FLeg(M\backslash \Delta_{\OT},\xi))$, where the homotopy group has a different base point, obtained from $A$ by several stabilizations such that
$$ \hat{A} = \pi_k(i) (\hat{B})$$ for some $\hat{B}\in \pi_k(\Leg(M\backslash \Delta_{\OT},\xi))$. Let $B\in \pi_k(\Leg(M\backslash \Delta_{\OT},\xi))$ be the unique Legendrian class satisfying
$$ \hat{B}=\pi_k(S^{t_{n}^{-}}_-\circ S^{t_{n}^{+}}_+\circ \cdots S^{t_{1}^{-}}_-\circ S^{t_{1}^{+}}_+)(B),$$ that exists because of Theorem \ref{thm:StabilizationMapsLooseLegendrians}. It remains to check that $\pi_k(i)(B)=A$. The only subtlety is that, a priori, we could lose control of the formal class during the previous process. However, we claim that this is not the case. Following \cite{CasalsDelPino}, the class $\hat{A}$ is obtained first by making the formal Legendrians (of a representative) of $A$ Legendrian near many points of the domain and adding several Legendrian stabilizations. Hence the claim follows from the following observation.
\begin{lemma}
    Let $K$ be a compact parameter space, $\varepsilon>0$ and $\delta^k:[-1-\varepsilon,1+\varepsilon]\rightarrow (M,\xi)$, $k\in K$, a family of Legendrian arcs. Assume that there exists two families of bypass embeddings $b_i^k\in \mathcal{B}(M,\xi)$, $(i,k)\in\{0,1\}\times K$, so that 
    \begin{itemize}
        \item The attaching arc of $b_i^k$ is $\delta^k[-1,1]$ for all $(i,k)\in\{0,1\}\times K$,
        \item $b_0^k$ and $b_1^k$ coincide near the attaching arc for all $k\in K$, and
        \item $\image(b_i^k)\cap \image(\delta^k)=\delta^k[-1,1]$ for all $(i,k)\in\{0,1\}\times K$.
    \end{itemize}
    Then, the families of Legendrian arcs $\beta^k_0$ and $\beta^k_1$ obtained as a destabilization of the family $\delta^k$ by means of $b^k_0$ and $b^k_1$, respectively, are formally Legendrian isotopic.
\end{lemma}
\begin{proof}
    The existence of the two families of bypass embeddings allows to reduce the statement to the non-parametric case and to a precise local model. In the latter the result is obvious.
\end{proof}
After this Lemma $\pi_k(i)(B)=A$ and, therefore, the map $\pi_k(i)$ is surjective. This concludes the argument.\qed

\bibliographystyle{plain}
\bibliography{main}
\end{document}